\documentclass[12pt,reqno]{amsart}

\usepackage{a4wide,amscd,amsmath,amssymb,amsthm}
\usepackage[all]{xy}

\theoremstyle{plain}
\newtheorem{thm}{Theorem}[section]
\newtheorem{prop}[thm]{Proposition}
\newtheorem{lem}[thm]{Lemma}
\newtheorem{dfnlem}[thm]{Definition and Lemma}
\newtheorem{cor}[thm]{Corollary}
\newtheorem{add}[thm]{Addendum}

\theoremstyle{definition}
\newtheorem{dfn}[thm]{Definition}
\newtheorem{eg}[thm]{Example}
\newtheorem{obs}[thm]{Remark}

\def\quo#1#2{\raisebox{.5ex}{{#1}\!}/\raisebox{-.5ex}{\!{#2}}}

\def\type#1{\mathbf{type}^D_{#1}}
\def\CW#1#2{\mathbf{CW}_{#1}^{#2}}
\def\coefD{\mathbf{Coef}^D}
\def\coef#1{\mathbf{Coef}^{#1}}
\def\twistD{\mathbf{Twist}^D}
\def\twist#1#2{\mathbf{Twist}_{#1}^{#2}}
\def\D{\mathcal D}
\def\G{\mathcal G}
\def\M{\mathcal M}

\def\A{\mathbf{A}}
\def\B{\mathbf{B}}
\def\Top{\mathbf{Top}^D_c}
\def\Topdash{\mathbf{Top}^{D'}_c}
\def\coneD{\mathbf{cone}^D}

\def\prinD{\mathbf{Prin}^D}

\def\gr#1{\mathbf{Gr}^{#1}}
\def\grd#1{\mathbf{Grd}^{#1}}
\def\cross#1#2{\mathbf{cross}_{#1}^{#2}}
\def\modD{\mathbf{mod}^D}

\def\modhat{\mathbf{mod \hat{\phantom{x}}}}
\def\mod#1{\mathbf{mod}({#1})}
\def\chainD{\mathbf{chain}^D}
\def\qud#1#2{\mathbf{quad}_{#1}^{#2}}
\def \rqud#1#2{\mathbf{rquad}_{#1}^{#2}}

\def\ob#1{\text{Ob}(#1)}

\def\Z{\ensuremath{\mathbb{Z}}}

\usepackage{graphicx}
\usepackage{ifpdf}
\ifpdf
   \fi

\title{Presentation of Homotopy Types under a Space}
\date{\today}
\begin{document}

\author{Hans-Joachim Baues and Beatrice Bleile}

\address{Max Planck Institut f{\"u}r Mathematik\\
Vivatsgasse 7\\
D--53111 Bonn, Germany}

\email{baues@mpim-bonn.mpg.de}

\address{Second author's home institution: School of Science and Technology\\
University of New England\\
NSW 2351, Australia}

\email{bbleile@une.edu.au}

\begin{abstract}
We compare the structure of a mapping cone in the category $\mathbf{Top}^D$ of spaces under a space $D$ with differentials in algebraic models like crossed complexes and quadratic complexes. Several subcategories of $\mathbf{Top}^D$ are identified with algebraic categories. As an application we show that there are exactly 16 essential self--maps of $S^2\times S^2$ fixing the diagonal. 
\end{abstract}

\subjclass[2000]{55-02, 55P05, 55P15, 55Q15, 55Q35, 55U35, 18B40}
\keywords{Presentation of spaces, suspensions and mapping cones under a space, attaching cells, groupoids, crossed modules, quadratic modules, self--maps of products fixing the diagonal}

\maketitle

\section*{Introduction}
In this paper we consider the homotopy theory of spaces and maps under a space, $D$, and homotopies relative $D$. In particular, we investigate the presentation of a space $X$ as mapping cone 
of a map $\partial_X$ under $D$. Certain relative CW--complexes $(X,D)$ are such mapping cones and the homotopy type of $X$ under $D$ is determined by the homotopy type of its presentation $\partial_X$ in the category of twisted maps. The main results describe algebraic models of these presentations in low dimensions corresponding to the boundary maps of crossed complexes and quadratic complexes, respectively. As we are working under a space, $D$, we obtain several new results which are known for the case where $D  = \ast$ is a point. As an application of quadratic complexes we show that there are exactly 16 essential self--maps of the product $S^2 \times S^2$ fixing the diagonal, and compute the monoid formed by these self--maps.

\section{Presentation of Spaces under $D$}\label{underD}
Given a space $D$, the objects of the category $\Top$ are pairs $X = (X, D)$, where $D \rightarrowtail X$ is a cofibration. Morphisms are maps under $D$, that is, commutative diagrams
\[\xymatrix{
& D \ar@{>->}[dl] \ar@{>->}[dr] & \\
X \ar[rr]_f && Y.}\]
The relative cylinder $I X$ is  determined by the pushout diagram
\[\xymatrix{
I \times D \ar[d] \ar[r]^p & D \ar[d] \\
I \times X \ar[r] & I X, }\]
where $I = [1,0]$ is the unit interval and $p$ is the projection onto the second factor. The maps $i_t: X\rightarrow I X$ are induced by $X \rightarrow I \times X, x \mapsto (t,x)$. A homotopy relative $D$ is a map $I X \rightarrow Y$ under $D$. Homotopy relative $D$ is a natural equivalence relation on the category $\Top$ and we write $[X, Y]^D$ for the set of homotopy classes of maps $X \rightarrow Y$ relative $D$. In other words, $[X, Y]^D$ is the set of morphims in the resulting homotopy category \quo{$\Top$}{$\simeq$}. Note that $\Top$ has initial object $\ast = (D,D)$, given by the identity on $D$. The category $\Top$ allows arbitrary coproducts with the coproduct of $X$ and $Y$ given by the adjunction space $X \vee Y = X \cup_D Y$. In fact, $\Top$ is an I--category and a cofibration category in the sense of \cite{AH}.

A map $f: D \rightarrow D'$ induces the push forward functor
\[ f_{\ast}: \quo{$\Top$}{$\simeq$} \longrightarrow \quo{$\Topdash$}{$\simeq$}\]
which takes $(X, D)$ to the pushout $X\cup_f D'$. The functor $f_{\ast}$ is an equivalence of homotopy categories if $f$ is a homotopy equivalence.

A \emph{based} object in $\Top$ is a cofibration $D \rightarrowtail A$ together with a retraction denoted by $0_A$. The \emph{cone}, $CA$, of the based object, $A$, is the pushout
\[\xymatrix{
A \ar[r]^{0_A} \ar[d]_{i_1} & D \ar[d] \\
IA \ar[r] & CA.}\] 
The \emph{suspension}, $\Sigma A$, of the based object, $A$, is the pushout
\[\xymatrix{
A \ar[r]^{0_A} \ar[d]_{i_0} & D \ar[d] \\
CA \ar[r] & \Sigma A.}\] 
Note that $\Sigma A$ is a again based with $0_{\Sigma A}$ induced by the composite $IA \stackrel {p}{\rightarrow} A \stackrel{0_A}{\rightarrow} D$, where $p$ is the projection of the cylinder. Hence the iterated suspension $\Sigma^n A$ is defined for $n \geq 0$. The map
\[\xymatrix{\mu: [0,1] \ar[r]^-{\approx} & [0,2] \ar[r] & [0,1] \cup_{\{1\}} ([1,2]/\{1,2\})}\]
induces the pinching map $CA \longrightarrow CA \vee \Sigma A$ which is a map under $A$. 

Given a based object $A$ and a map $w: A \rightarrow W$ in $\Top$, the \emph{mapping cone} $C_w$ is the pushout
\[\xymatrix{
A \ar[r]^{w} \ar[d] & W \ar[d] \\
CA \ar[r] & C_w.}\] 
The pinching map of $CA$ induces a \emph{coaction map}
\[\mu_{C_w}: C_w \rightarrow C_w \vee \Sigma A,\]
and, for maps $f: C_w \rightarrow Z$ and $\alpha: \Sigma A \rightarrow Z$ in $\Top$, we write $f+\alpha = (f, \alpha)\mu_{C_w}$ to denote the induced action. 

A \emph{$D$--cone} is a mapping cone $X = C_w$, as above, with $W = D$; in this case we write $X' = \Sigma A$ and $X \rightarrow X \vee X'$ for the coaction map of the $D$--cone $X$. A $D$--cone $C_w$ is a suspension if $w = 0_A$ is the trivial map of $A$. We denote the full subcategory of $D$--cones in $\quo{$\Top$}{$\simeq$}$ by $\coneD$. The category $\coneD$ is a theory of coactions as considered in \cite{AH}. If $X''$ is a suspension, a map $\partial_X: X'' \rightarrow X$ in $\coneD$ is called a \emph{presentation}. Such presentations are the objects in the category $\twistD$. Morphisms, $\partial_X \rightarrow \partial_Y$, in $\twistD$ are commutative diagrams 
\[\xymatrix{
X'' \ar[r]^-{f''} \ar[d]_{\partial_X} & Y'' \vee Y \ar[d]^{(\partial_Y,1)}  \\
X \ar[r]_f & Y}\]
in $\coneD$ with $(0_{Y''}, 1) f''= 0_{X''}$. Composition is defined by $(f'',f) (g'',g) = (\overline{f}g'',fg)$, where $\overline{f} = (f'', i_Y f): X'' \vee X \rightarrow Y'' \vee Y$. The category $\prinD$ is the subcategory of \emph{principal maps} in $\twistD$ consisting of presentations, $\partial_X: X'' \rightarrow X$, and morphisms, $(f'',f)$, given by commutative diagrams
\[\xymatrix{
X'' \ar[r]^-{f''} \ar[d]_{\partial_X} & Y'' \ar[d]^{\partial_Y}  \\
X \ar[r]_f & Y}\]
in $\coneD$.

Given presentations, $\partial_X$ and $\partial_Y$, the map $f: X \rightarrow Y$ in $\coneD$ is \emph{$\partial$--compatible}, if there is a map $f'': X'' \rightarrow Y'' \vee Y$, such that $(f'', f)$ is a morphism in $\twistD$. Two maps $f, f_1: X \rightarrow Y$ are \emph{$\partial$--equivalent}, $f \sim f_1$, if there is a map $\alpha: X' \rightarrow Y''\vee Y$, with $(0_{Y''}, 1)\alpha = 0_{X'}$, such that
\[f_1 = f + (\partial_Y,1)\alpha.\]
The objects of the category $\coefD$ are the same as the objects of $\twistD$, namely presentations $\partial_X: X'' \rightarrow X$. Morphisms in $\coefD$ are $\partial$--equivalence classes, $\{f\}: \partial_X \rightarrow \partial _Y$, of $\partial$--compatible maps, $f: X \rightarrow Y$, with composition given by $\{f\}\{g\} = \{fg\}$. Identifying the $D$--cone $X$ with the presentation $\ast \rightarrow X$, where $\ast$ is the initial object in $\Top$, we obtain the full inclusion of categories
\[\coneD \subset \coefD.\]

Take a presentation $\partial_X: X'' \rightarrow X$ in $\coneD$, choose a representative $\partial_X'$ of the homotopy class $\partial_X$ and denote the mapping cone of $\partial_X'$ by $X_2$. Then there is an inclusion $i: X \rightarrow X_2$ and $\partial_X$ is called a \emph{presentation of the homotopy type} of $X_2$. 

\begin{obs}\label{maponX_2}
Given a morphism $(f'', f): \partial_X \rightarrow \partial_Y$ in $\twistD$, there is a \emph{twisted map} $\overline f: X_2 \rightarrow Y_2$ associated with $(f'',f)$, we refer the reader to V.2.3 in Baues \cite{AH} for the construction of $\overline f$. Thus a $\partial$--compatible map, $f: X \rightarrow Y$ in $\coneD$, yields a twisted map $\overline f: X_2 \rightarrow Y_2$ extending $f$, that is, $i_{\ast}f = i^{\ast} \overline f$. If $f \sim f_1$ are $\partial$--equivalent, then $if \simeq if_1$ relative $D$.
\end{obs}

A relative CW--complex, $X = (X, D)$, under a space $D$ is given by skeleta, $(X^n, D)$, obtained inductively by attaching $n$--cells to $(X^{n-1}, D)$, where $n \geq 0$ and $X^{-1} = D$. We say that $(X, D)$ is \emph{pointed}, if $D$ has a base point $\ast$, $X^0 = D$, all attaching maps preserve base points and $1$--cells are attached via the map $S^0 \rightarrow \ast$. We denote the full subcategories of $\quo{$\Top$}{$\simeq$}$ consisting of relative and pointed relative CW--complexes by $\CW{}{D}$ and $\CW{}{D, \ast}$, respectively. Note that every relative CW--complex, $(X, D)$, for which $D$ and $D \rightarrowtail X$ are connected, is equivalent in $\quo{$\Top$}{$\simeq$}$ to a pointed relative CW--complex.

\section{Attaching $n$--cells}\label{ncells}
Let $\CW{n}{D}$ ($\CW{n}{D,\ast}$) be the full subcategory of $\Top/\simeq$ consisting of (pointed) relative CW--complexes $X = (X,D)$ with cells in dimension $n$ only.
\begin{lem}\label{CWnDisfull}
For $n \geq 1, \CW{n}{D}$ is a full subcategory of the category $\coneD$ of $D$--cones. Up to equivalence of categories, the category $\CW{n}{D}$ depends on the $n$--type of $D$ only. The same holds for $\CW{n}{D,\ast}$.
\end{lem}
\begin{proof}
Consider based objects under $D$ given by the disjoint union of $D$ and a discrete set, $A_0$, so that the trivial map, $0_A$, is determined by a map $d: A_0 \rightarrow D$. The iterated suspension is the adjunction space, $\Sigma^{n-1}A = A_0 \times S^{n-1} \cup_d D$, with trivial map determined by $A_0 \times S^{n-1} \rightarrow A_0 \stackrel{d}{\rightarrow} D$. Given $X = (X,D)$ in $\CW{n}{D}$, let $A_0$ be the set of $n$--cells of $(X,D)$. Then the attaching maps in $(X,D)$ yield a map $w: \Sigma^{n-1}A \rightarrow D$, such that $X = C_w$. Now let $f: D \rightarrow D'$ be the $n$--type of $D$. Then we may assume that we obtain $D'$ from $D$ by attaching cells of dimension $\geq n+2$. Thus the push forward functor yields an equivalence $\CW{n}{D} \stackrel{\sim}{\rightarrow} \CW{n}{D'}$ by the cellular approximation theorem.
\end{proof}

\begin{eg}\label{egast}
Let $D = \ast$ be a point. A presentation $\partial = \partial_X$ in $\CW{1}{\ast}$ is a map between $1$--point unions of $1$--spheres, and hence corresponds to a homomorphism $\partial_X: X'' \rightarrow X$ of free groups. Such a homomorphism is a presentation of the group $G = \quo{$X$}{$N(\partial_X(X''))$}$, where $N(\partial_X(X''))$ is the normal subgroup in $G$ generated by $\partial_X(X'')$. We leave it to the reader to verify that $\partial$--equivalence classes of $\partial$--compatible maps, $\partial_X \rightarrow \partial_Y$, between such presentations $\partial_X$ of the group $G$ and $\partial_Y$ of the group $H$, respectively, correspond to group homomorphisms $G \rightarrow H$. We conclude that the set of morphisms $\partial_X \rightarrow \partial_Y$ in $\coef{\ast}$ coincides with the set of group homomorphisms $G \rightarrow H$, and that the full subcategory of $\coef{\ast}$ of presentations in $\CW{1}{\ast}$ is equivalent to the category of groups.
\end{eg}

\begin{eg}\label{egdiscrete}
For a discrete set, $D$, the full subcategory of $\coefD$ consisting of presentations in $\CW{1}{D}$ is equivalent to the category $\grd{D}$ of groupoids with object set $D$ and morphisms given by functors which are the identity on $D$. 
\end{eg}

We denote the categories of groups and groupoids by $\gr{}$ and $\grd{}$, respectively. For a group, $G$, and a groupoid, $\G$, let $\gr{G}$ and $\grd{\G}$ be the categories of objects under $G$ and $\G$, respectively. Free objects in $\gr{G}$ and $\grd{\G}$ are also called \emph{cofibrant} objects and we denote the subcategories of cofibrant objects by $\gr{G}_c$ and $\grd{\G}_c$, respectively. Note that an object $G \rightarrow H$ in $\gr{G}$ is cofibrant if $H$ is a coproduct in $\gr{}$ of $G$ and a free group.

\begin{prop}
Let $\G$ be a groupoid representing the $1$--type of $D$. Then there is an equivalence of categories
\[ \tau: \CW{1}{D} \stackrel{\sim}{\longrightarrow} \grd{\G}_c.\]
If $D$ is connected and pointed, let $G = \pi_1(D)$ be the fundamental group. Then there is a corresponding equivalence of categories 
\[ \tau: \CW{1}{D,\ast} \stackrel{\sim}{\longrightarrow} \gr{G}_c.\]
\end{prop}

\begin{proof}
Use the van Kampen theorem in its general form \cite{BHS}.
\end{proof}

\section{Crossed Modules}
To describe algebraic categories equivalent to $\CW{2}{D,\ast}$ and $\CW{2}{D}$, we recall the definition of precrossed and crossed modules over groups and then that of crossed modules over groupoids. 

A \emph{pre--crossed module} $M$ (over the group $M_1$) is a group homomorphism $d= d_{M}:  M_2 \rightarrow M_1$, where $M_1$ acts on $M_2$ from the right, such that $d(x^{m}) = - m + d(x) + m$, for $m \in M_1$ and $x \in M_2$. The \emph{Peiffer commutator} of $x, y \in M_2$ is given by
\[\langle x, y \rangle = - x - y + x + y^{d(x)},\]
and $M$ is a \emph{crossed module} if all Peiffer commutators are trivial. A \emph{morphism between crossed modules} $M = (d_{M}: M_2 \rightarrow M_1)$ and $N = (d_{N}: N_2 \rightarrow N_1)$ is given by a commutative diagram
\[\xymatrix{
M_2 \ar[d]_{q_2} \ar[r]^{d_{M}} & M_1 \ar[d]^{q_1} \\
N_2 \ar[r]_{d_N} & N_1,}\]
in the category of groups, where $q_2$ is $q_1$--equivariant. 

Given a groupoid, $P$, we write $\ob{P}$ for its set of objects. For $p, q \in \ob{P}$, we write $P(p,q)$ for the set of morphisms from $p$ to $q$ and $P(p)$ for $P(p,p)$. A \emph{crossed module over the groupoid $M_1$} is a morphism of groupoids $d_M: M_2 \rightarrow M_1$, such that
\begin{enumerate}
\item $M_2$ is a totally disconnected groupoid with the same objects as $M_1$. Equivalently, $M_2$ may be viewed as a family of groups $\{ M_2(m) \}_{m \in \ob{M_1}}$. We use additive notation for all groups $M_2(m)$;
\item $d_M$ is a family of group homomorphisms $\{ d_M(m): M_2(m) \rightarrow M_1(m) \}_{m \in \ob{M_1}}$;
\item The groupoid $M_1$ acts on $M_2$ from the right. If $x \in M_2(m)$ and $a \in M_1(m,n)$, then $x^a \in M_1(n)$ and the usual axioms of an action are satisfied, that is, $x^{ab} = (x^a)^b$ and $(x+y)^a = x^a+y^a$ for $x, y \in M_2(m)$ and $a, b \in M_1(m, n)$;
\item $d_M$ preserves actions, where  $M_1$ acts on itself by conjugation, that is, $d_M(x^a) = -a + d_M(x) + a$ for all $x \in M_2(m)$ and $a \in M_1(m,n)$ and
\item All Peiffer commutators are trivial, that is, $y^{d_M(x)} = - x + y + x$ for all $x, y \in M_2(m)$.
\end{enumerate}

A \emph{morphism of crossed modules (over groupoids)}, $f: M \rightarrow N$, is a pair of morphisms $f_i: M_i \rightarrow N_i, i = 1, 2$, inducing the same map of objects and compatible with the boundary maps and the actions of both crossed modules.

Take a crossed module $\D = (d_{\D}: \D_2 \rightarrow \D_1)$ over groupoids and let $\cross{2}{\D}$ be the category of crossed modules $M$ under $\mathcal D$ of the form
\[\xymatrix{
\D_2 \ar[d]_{q_2} \ar[r]^{d_{\D}} & \D_1 \ar@{=}[d]^{q_1}  \\
M_2 \ar[r]_{d_{M}} & M_1,}\]
where $q_1$ is the identity. If $\D = (d_{\D}: \D_2 \rightarrow \D_1)$ is a crossed module over groups, we denote the corresponding category of crossed modules over groups under $\D$ by $\cross{2}{\D,\ast}$.

Now we compute the category $\CW{2}{D}$ using \emph{cofibrant} objects in $\cross{2}{\D}$. The object $M$ under $\D$ in $\cross{2}{\D}$ is \emph{cofibrant} if there is a set, $Z_2$, of generators together with a function, $j: Z_2 \rightarrow M_2$, such that, for every crossed module $M'$ under $\D$ and every map $a_Z: Z_2 \rightarrow M'_2$ with $d_{M'}a_Z = d_Mj$, there is a unique map $a: M \rightarrow M'$ under $\D$ such that $a_2 j = a_Z$. We denote the full subcategory of $\cross{2}{\D}$ consisting of cofibrant objects by $\cross{2,c}{\D}$. Similarly, we obtain the category $\cross{2,c}{D,\ast}$ for crossed modules over groups under $\D$.

\begin{prop}\label{CW2dcrossceq}
Let $\D$ be a crossed module over groupoids representing the $2$--type of $D$ with $\D_1$ a free groupoid. Then there is an equivalence of categories
\[ \tau: \CW{2}{D} \stackrel{\sim}{\longrightarrow} \cross{2,c}{\D}.\]
If $D$ is connected and pointed, there is a corresponding equivalence of categories
\[ \CW{2}{D,\ast} \stackrel{\sim}{\longrightarrow} \cross{2,c}{\D,\ast},\]
where $\D_1$ is a free group.
\end{prop}

\begin{proof}
We may assume that $D$ is a CW--complex such that the crossed module $\D$ over groupoids representing the $2$--type of $D$ is given by $\pi_2(D, D^1) \stackrel{\partial}{\rightarrow} \pi_1(D^1)$, where $D^1$ is the $1$--skeleton of $D$. Then the functor $\tau$ takes $(Y,D)$ to $\pi_2(Y,D^1) \rightarrow \pi_1(D^1)$, which is a cofibrant crossed module under $\D$ by a theorem of J.H.C. Whitehead \cite{Whitehead3}.
\end{proof}

\section{Quadratic Modules}\label{quadmod}
In this section we describe an algebraic category equivalent to the full subcategory of pointed relative CW--complexes $(X, D)$ with cells in dimensions 2 and 3 only, where $D$ is connected. In general, we say the \emph{dimension}, $\text{dim}(X, D)$, of the relative CW--comlex $(X, D)$ is less than or equal to $n$, if $X = X^n$. Further, $X$ has \emph{trivial $k$--skeleton}, if $X^k = D$. We denote the full subcategory of \quo{$\Top$}{$\simeq$} consisting of relative CW--complexes $(X, D)$ of dimension less than or equal to $n$ with trivial $k$--skeleton by $\CW{k,n}{D}$. To describe the category $\CW{1,3}{D,\ast}$ of pointed relative CW--complexes with cells in dimensions 2 and 3 only, we recall the definition of quadratic modules, see \cite{4dim}. As we are working with pointed relative CW--complexes $(X, D)$, where $D$ is connected, we restrict attention to pre--crossed, crossed and quadratic modules over groups for the remainder of this section. 

Given a pre--crossed module, $d_M: M_2 \rightarrow M_1$, the \emph{Peiffer subgroup, $P_2(d_M)$,} of $d_M$ is the subgroup of $M_2$ generated by Peiffer commutators $\langle x, y \rangle$. Peiffer subgroups generalize commutator subgroups and $P_2(d_M)$ is a normal subgroup of $M_2$. The quotient map $q: M_2 \rightarrow \quo{$M_2$}{$P_2(d_M)$}$ is equivariant with respect to the action of $M_1$ and $d^{cr}_M: M_2^{cr} = \quo{$M_2$}{$P_2(d_M)$} \rightarrow M_1$ is the \emph{crossed module associated with} $d_M$. 

A pre--crossed module $d_M$ is \emph{Peiffer niplotent of class} $2$ or a \emph{$\text{nil}(2)$--module}, if all iterated Peiffer commutators of length $\geq 3$ vanish. Given a $\text{nil}(2)$--module $d_M$, let $C = (M_2^{cr})^{ab}$ be the abelianization of the associated crossed module and let $\{x\} \in C$ denote the class of $x \in M_2$. Then $M_1$ acts on $C \otimes C$ via $(\{x\} \otimes \{y\})^m = \{x^m\} \otimes \{y^m\}$ and 
\[w(\{x\} \otimes \{y\}) = \langle x,y \rangle  = -x-y+x+y^{d_Mx}\]
defines an equivariant homomorphism $w: C \otimes C \longrightarrow M$, called the \emph{Peiffer commutator map}, see IV 1 in \cite{4dim}.
A \emph{quadratic module}, $Q = (\omega, \partial_3, \partial_2)$, is a diagram
\[\xymatrix{
& C \otimes C \ar[dl]_{\omega} \ar[d]^w & \\
Q_3 \ar[r]_{\partial_3} & Q_2 \ar[r]_{\partial_2} & Q_1,}\]
in $\gr{}$, such that
\begin{enumerate}
\item The homomorphism $\partial_2: Q_2 \rightarrow Q_1$ is a $\text{nil}(2)$--module with $C = (Q_2^{cr})^{ab}$ and Peiffer commutator map $w$;
\item The \emph{boundary homomorphisms} $\partial_3$ and $\partial_2$ satisfy $\partial_2\partial_3 = 0$, and the \emph{quadratic map}, $\omega$, is a lift of the Peiffer commutator map, that is, $\partial_3 \omega = w$;
\item $Q_1$ acts on $Q_3$, all homomorphisms are equivariant with respect to the action of $Q_1$ and, for $q \in Q_3, x \in Q_2$, \[q^{\partial_2x} = q + \omega(\{\partial_3q\}\otimes\{x\} + \{x\} \otimes \{\partial_3q\}) \quad \text{and}\]
\item Commutators in $Q_3$ satisfy $(p,q) = -p-q+p+q = \omega(\{\partial_3p\}\otimes\{\partial_3q\})$, for $p, q \in Q_3$.
\end{enumerate}
A \emph{morphism}, $f = (f_3, f_2, f_1): Q \rightarrow Q'$, of quadratic modules is given by a commutative diagram
\[\xymatrix{
C \otimes C \ar[r]^{\omega} \ar[d]_{f_{\ast}} & Q_3 \ar[r]^{\partial_3} \ar[d]_{f_{3}} & Q_2 \ar[r]^{\partial_2} \ar[d]_{f_{2}} & Q_1 \ar[d]^{f_{1}} \\
C' \otimes C' \ar[r]_{\omega'} & Q'_3 \ar[r]_{\partial'_3} & Q'_2 \ar[r]_{\partial'_2} & Q'_1,}\]
where $(f_2, f_1)$ is a morphism of pre--crossed modules inducing $f_{\ast}: C \rightarrow C'$ and $f_3$ is an $f_1$--equivariant homomorphism. Let $\qud{}{}$ be the category of quadratic modules and  define \emph{cofibrations} in $\qud{}{}$ via free extensions as on p. 205 in V.1 of \cite{4dim}. Given a quadratic module $\D$, we denote the category of quadratic modules under $\D$ by $\qud{}{\D}$ and the subcategory of cofibrant objects, $q: \D \rightarrowtail Q$ in $\qud{}{\D}$ with $q_1: D_1 \rightarrow Q_1$ the identity, by $\qud{c}{\D}$. There is a notion of homotopy for morphisms in $\qud{}{}$ defined in \cite{4dim}. Homotopies in $\qud{c}{\D}$ are relative $\D$, that is, they are given by maps $\alpha: Q_1 \rightarrow Q_2$ with $\alpha |_{Q_1} = 0$, and we denote the resulting homotopy category by $\quo{$\qud{c}{\D}$}{$\simeq$}$. As a consequence of Theorem IV.8.1 in \cite{4dim} we obtain

\begin{prop}\label{CW13}
Let $\D$ be a quadratic module representing the $3$--type of the connected pointed space $D$ and assume that $\D$ is free in degrees $1$ and $2$. Then there is an equivalence of categories
\[ \CW{1,3}{D,\ast} \stackrel{\sim}{\longrightarrow} \quo{$\qud{c}{\D}$}{$\simeq$}.\]
\end{prop}

As quadratic modules over groupoids are not yet understood, no such result is available for $\CW{1,3}{D}$ in the non--pointed case. See \cite{Ellis} for further topological interpretation of quadratic modules.

\section{$n$--Types under $D$}
An \emph{$n$--type} is a space $X$ with trivial homotopy groups $\pi_i(X, x_0)=0$, for $i > n$ and all base points $x_0 \in X$. Given a space $D$, an \emph{$n$--type under $D$} is a cofibration $j: D \rightarrowtail X$, where $X$ is an $n$--type and the inclusion is an $(n-1)$--equivalence, that is, the induced map $\pi_i(j): \pi_i(D)\rightarrow \pi_i(X)$ is a surjection for $i = n-1$ and an isomorphism for $i < n-1$. The following result describes the category $\type{n}$ of $n$--types under $D$.

\begin{lem}\label{coefeq}
For $n \geq 1$, the full subcategory of $\coefD$ consisting of presentations $\partial_X$ in $\CW{n}{D}$ is equivalent to the category $\type{n}$. If $D$ is connected and pointed, the full subcategory of $\coefD$ consisting of presentations $\partial_X$ in $\CW{n}{D, \ast}$ is equivalent to the category $\type{n}$.
\end{lem}
\begin{proof}
Use Remark \ref{maponX_2}.
\end{proof}

For Examples \ref{egast} and \ref{egdiscrete}, where $n = 1$, it is well-known that groups and groupoids correspond to $1$--types, respectively. In general, we obtain

\begin{prop}\label{type1eq}
Let $\D$ be a groupoid representing the $1$--type of $D$. Then there is an equivalence of categories
\[ \type{1} \stackrel{\sim}{\longrightarrow} \grd{\D}.\]
If $D$ is connected and pointed, let $G = \pi_1(D)$ be the fundamental group. Then there is a corresponding equivalence of categories
\[ \type{1} \stackrel{\sim}{\longrightarrow} \gr{G}.\]
\end{prop}

For $n = 2$ we obtain an algebraic description of $\type{2}$ in terms of crossed modules over groupoids. 

\begin{prop}\label{type2eq}
Let $\mathcal D$ be a crossed module representing the $2$--type of $D$ with $\D_1$ a free groupoid. Then there is an equivalence of categories
\[ \type{2} \stackrel{\sim}{\longrightarrow} \cross{2}{\D}.\]
If $D$ is connected and pointed, there is a corresponding equivalence of categories
\[ \type{2} \stackrel{\sim}{\longrightarrow} \cross{2}{\D,\ast}.\]
\end{prop}

By Theorem \ref{coefeq}, the result above computes certain subcategories of the category $\coefD$, thus generalizing Examples \ref{egast} and \ref{egdiscrete}. Further, we obtain

\begin{prop}\label{type3eq}
Let $\D$ be a quadratic module representing the $3$--type of $D$ with $\D$  free in degrees $1$ and $2$. Then there is an equivalence of categories
\[\type{3} \stackrel{\sim}{\longrightarrow} \qud{3}{\D},\]
where the right hand side is the category of quadratic modules, $q: \D \rightarrow Q$, under $\D$, such that $q_1$ and $q_2$ are the identity morphisms.
\end{prop}

Propositions \ref{type1eq}, \ref{type2eq} and \ref{type3eq} follow from Lemma \ref{coefeq} by considering the algebraic categories in Sections \ref{underD} and \ref{ncells}.

\section{Modules and Chain Complexes}
In this section we use homotopy groups and partial suspensions in the category $\Top$ to introduce the category of modules associated with presentations in $\coneD$. We recall the notation from Baues \cite{AH}.

For a based object $A$ and an object $U$ in $\Top$, we consider the \emph{homotopy groups} 
\[\pi_n^A(U) = [\Sigma^n A,U]^D, n \geq 0,\]
of homotopy classes relative $D$ in $\Top$. The suspension, $\Sigma^n(CA, A)$, of $(CA, A)$ is the pair $(\Sigma^n CA, \Sigma^n A)$ and the \emph{relative homotopy groups} of a pair $(U,V)$ in $\Top$ are given by the set of homotopy classes of pair maps,
\[\pi_{n+1}^A(U,V) = [\Sigma^n (CA,A),(U,V)]^D, n \geq 0.\]
For $n\geq 0$, we obtain the exact sequence
\begin{equation}\label{pilongex}
\ldots \longrightarrow \pi_{n}^A(U) \stackrel{j}{\longrightarrow} \pi_{n}^A(U,V) \stackrel{\partial}{\longrightarrow} \pi_{n-1}^A(V) \stackrel{i}{\longrightarrow} \pi_{n-1}^A(U) \longrightarrow \ldots
\end{equation}
which is an exact sequence of groups if $A$ is a cogroup in $\quo{$\Top$}{$\simeq$}$, see II.7.8. \cite{AH}. 

Given a based object $B$ and an object $Y$ in $\Top$, the retraction $(0,1): B \vee Y \rightarrow Y$ induces \[(0,1)_{\ast}: \pi_n^A(B \vee Y) \rightarrow \pi_n^A(Y).\] We put
\[\pi_n^A(B \vee Y)_2 = [\Sigma^nA, B \vee Y]^D_2 = \ker \big( (0,1)_{\ast}\big).\]
If $A$ is a suspension in $\coneD$, the operators $j$ and $\partial$ in the exact sequence (\ref{pilongex}) yield the isomorphisms $j$ and $\partial$ in the diagram
\[\xymatrix{
& \pi_n^A(CB \vee Y, B \vee Y)\ar[r]^-{\partial}_-{\cong} \ar[d]^{(\pi_0\vee 1_Y)_{\ast}} & \pi_{n-1}^A(B \vee Y)_2\\
\pi_{n}^A(\Sigma B \vee Y)_2 \ar[r]^-{j}_-{\cong} & \pi_n^A(\Sigma B \vee Y, Y),}\]
where $\pi_0: CA \rightarrow \Sigma A$ is the map in the pushout diagram defining $\Sigma A$. We define the \emph{partial suspension} 
\[E: \pi_{n-1}^A(B \vee Y)_2 \longrightarrow \pi_{n}^A(\Sigma B \vee Y)_2\]
by $E = j^{-1} (\pi_0 \vee 1_Y)_{\ast} \partial^{-1}$. For $Y = \ast$, this is the suspension
\[\Sigma: \pi_{n-1}^A(B) \longrightarrow \pi_n^A(\Sigma B).\]
For a proof that $E$ and $\Sigma$ are group homomorphisms, we refer the reader to II.11 in \cite{AH}.

\begin{dfn}
The objects of the \emph{category of modules}, $\modD$,  associated with $\coefD$ are coproducts $A \vee \partial_X$, where $A$ is a suspension, considered as an object of $\coefD$ via the inclusion $\coneD \subset \coefD$, and $\partial_X$ is a presentation. Given a morphism, $u: \partial_X \rightarrow \partial_Y$, in $\coefD$, a \emph{$u$--equivariant map} is a pair $(w, u)$, written as
\[w \odot u: A \vee \partial_X\longrightarrow B \vee \partial _Y, \]
where $w$ is an element in the image of the composition
\[ [A, B \vee Y]^D_2 \stackrel{E}{\longrightarrow} [ \Sigma A, \Sigma B \vee Y ]^D_2  \stackrel{i_{\ast}}{\longrightarrow} [ \Sigma A, \Sigma B \vee Y_2]^D_2.
\]
Here $Y_2$ is the mapping cone of a representative $\partial_Y': Y'' \rightarrow Y$ of the homotopy class $\partial_Y$ and $i_\ast$ is induced by the inclusion $Y \subset Y_2$. Morphisms in $\modD$ are $u$--equivariant maps $(w, u) = w \odot u$ and composition is given by
\[(w \odot u)(w' \odot u') = (w, i_2u)w' \odot uu',\]
where $i_2u: X_2 \rightarrow Y_2 \subset \Sigma B \vee Y_2$ is given by a map $X_2 \rightarrow Y_2$ extending a representative $X \rightarrow Y$ of $u$. We say that $w \odot u$ is \emph{$\partial_X$--equivariant} if $u$ is the identity on $\partial_X$.
\end{dfn}

There is a canonical \emph{coefficient functor}
\[c: \modD \longrightarrow \coefD\]
which maps objects $A \vee \partial_X$ to $\partial_X$ and morphisms $w \odot u$ to $u$.

\begin{eg}
Consider the category $\modhat$ of free modules over groups. The objects of $\modhat$ are pairs $(G, M)$, where $G$ is a group and $M$ a free right $G$--module. Morphisms are pairs $(\varphi, f): (G, M) \rightarrow (G', M')$, where $\varphi: G \rightarrow G'$ is a group homomorphism and $f: M \rightarrow M'$ a $\varphi$--equivariant homomorphism. A group, $G$, has a presentation $\partial_G$ in $\coefD$, where $D = \ast$ is a point, as in Example \ref{egast}. For $n \geq 2$, the full subcategory of $\modD$ consisting of objects, $A \vee \partial_X$, where $A$ is a wedge of $n$--spheres, is equivalent to the category $\modhat$. The coefficient functor, $c$, corresponds to the functor $\modhat \rightarrow \mathbf{Gr}$ which maps $(G, M)$ to $G$.
\end{eg}

The subcategory, $\mod{G}$, of $G$--equivariant maps in  $\modhat$ is an additive category. Similarly, the full subcategory $\mod{\partial_X}$ of $\modD$ consisting of $\partial_X$--equivariant maps is an additive category with sum
\[(A, \partial_X) \oplus (B, \partial_X) = A \vee B \vee \partial_X.\] 

\begin{dfn}
The objects of the category $\chainD$ of \emph{chain complexes} in $\modD$ are pairs $(A, \partial_X)$, where $\partial_X$ is an object in $\coefD$ and $A$ is a chain complex in the additive category $\mod{\partial_X}$, that is, a sequence of suspensions, $A_i, i \in \Z$, in $\coneD$ together with $\partial_X$--equivariant maps
\[ d_i: A_i \vee \partial_X \longrightarrow A_{i-1} \vee \partial_X\]
in $\modD$, such that $d_{i-1}d_i = 0$ in the abelian group of $\partial_X$--equivariant maps. A morphism $(f,u) : (A, \partial_X) \rightarrow (B, \partial_Y)$ in $\chainD$ consists of a morphism $u: \partial_X \rightarrow \partial_Y$ in $\coefD$ and a $u$--equvariant chain map $f$, that is, a sequence $f = (f_i)_{i \in \Z}$ such that the diagram
\[\xymatrix{
A_i \vee \partial_X \ar[d]_{f_i} \ar[r]^{d_i} & A_{i-1} \vee \partial_X \ar[d]^{f_{i-1}} \\
B_i \vee \partial_Y \ar[r]^{d_i} & B_{i-1} \vee \partial_Y} \]
commutes. Two morphisms, $(f,u), (g,v) : (A, \partial_X) \rightarrow (B, \partial_Y)$ are \emph{homotopic} in $\chainD$, $(f,u) \simeq (g,v)$, if $u = v$ and if there is a sequence, $\alpha = (\alpha_i)_{i \in \Z}$, of $u$--equivariant maps
\[\alpha_i: A_i \vee \partial_X \longrightarrow B_{i+1} \vee \partial_Y\]
in $\modD$, such that 
\[-f_i + g_i = d_{i+1}\alpha_i + \alpha_{i-1}d_i.\]
\end{dfn}

\section{The Homotopy Category of Twisted Maps}
Given a based object $A$ and a map $w: A \rightarrow D$ under $D$, the inclusions $i_1:  \Sigma A \rightarrow \Sigma A \vee C_w$ and $i_2: C_w \rightarrow \Sigma A \vee C_w$ yield the map
\[ i_2+i_1: C_w \longrightarrow \Sigma A \vee C_w \quad \text{in} \ \coneD.\]
Except for the order of the objects in the wedge, this is the coaction $\mu_{C_w}$. For $f: \Sigma X \rightarrow C_w$, the \emph{difference element}
\[\nabla f: = -f^{\ast}(i_2) + f^{\ast}(i_2 + i_1) : \Sigma X \longrightarrow \Sigma A \vee C_w\]
is trivial on $C_w$, that is, $(0,1)_{\ast} \nabla f = 0$, and thus yields the \emph{difference operator}
\[ \nabla: \pi_1^X(C_w) \longrightarrow \pi_1^X(\Sigma A \vee C_w)_2.\]

\begin{dfnlem}\label{Kdfn}
For an object $\partial_X: X'' \rightarrow X$ in $\twistD$, the map
\[d_X = E(\nabla \partial_X) \odot 1: X'' \vee \partial_X \longrightarrow X' \vee \partial_X\]
is a $\partial_X$--equivariant map in $\modD$ which we view as a chain complex concentrated in degrees $1$ and $2$. To a map $(f'',f): \partial_X \rightarrow \partial_Y$ in $\twistD$ we assign the map $(f_1, f_2) = (E(\nabla f) \odot u, E(f'') \odot u)$, where $u: \partial_X \rightarrow \partial_Y$ in $\coefD$ is represented by $f$. 
\[\xymatrix{
X'' \vee \partial_X \ar[r]^{d_X} \ar[d]_{f_2=E(f'') \odot u} & X' \vee \partial_X \ar[d]^{f_1=E(\nabla f) \odot u} \\
Y'' \vee \partial_Y \ar[r]^{d_Y} & Y' \vee \partial_Y }\]
Then
\begin{eqnarray*}
K: \twistD & \longrightarrow & \chainD \\
\partial_X & \longmapsto & d_X \\
(f'',f): \partial_X \rightarrow \partial_Y & \longmapsto & (f_1, f_2)
\end{eqnarray*}
is a well--defined functor. 

Further, there is a natural equivalence relation, $\simeq_E$, on $\twistD$, namely, $(f'', f) \simeq_E (g'', g): \partial_X \rightarrow \partial_Y$, if there is a map $\alpha: X' \rightarrow Y'' \vee Y$ with $(0_{Y''}, 1_Y) \alpha = 0_{X'}$, such that 
\[g = f + (\partial_Y, 1_Y)\alpha\]
and
\[ -E(f'') \odot u +E(g'') \odot u = E(\alpha) \odot u d_X\]
in $\modD(X'' \vee \partial_X, Y'' \vee \partial_Y)_u$, where $u = \{f\} = \{g\}$.
\end{dfnlem}

\begin{lem}
The functor $K$ in Lemma \ref{Kdfn} induces a well--defined functor 
\[K: \quo{$\twistD$}{$\simeq_E$} \longrightarrow \quo{$\chainD$}{$\simeq$}\]
between homotopy categories which is compatible with coefficient functors.
\end{lem}

The properties of the functor $K$ are described in \cite{CFHH}.
\section{Homotopy Types of Relative CW--Complexes under $D$}
Recall that $\CW{k,n}{D}$ is the full subcategory of \quo{$\Top$}{$\simeq$} consisting of relative CW--complexes $(X, D)$ of dimension less than or equal to $n$ with trivial $k$--skeleton.

\begin{lem}\label{k,k+r}
For $k \geq r - 1 \geq 0$, every relative CW--complex $Y = (Y, D)$ in $\CW{k, k+r}{D}$ is homotopy equivalent under $D$ to a $D$--cone $X = C_w$, where $w: A \rightarrow D$ and the based object $A$ in $\Top$ is given by a relative CW--complex $(A, D)$ in $\CW{k-1, k+r-1}{D}$.
\end{lem}

\begin{proof}
See 2.3.5 in \cite{OT}.
\end{proof}

\begin{lem}\label{k,2r+k}
For $k \geq r - 1 \geq 0$, every relative CW--complex $Y = (Y, D)$ in $\CW{k, 2r+k}{D}$ is homotopy equivalent under $D$ to the mapping cone of a presentation $\partial_X: X'' \rightarrow X$ in $\coneD$, where
\begin{itemize}
\item[(i)] $X$ is a $D$--cone as in Lemma \ref{k,k+r} which is homotopy equivalent under $D$ to the $(k+r)$--skeleton $Y^{k+r}$, and
\item[(ii)] $X''$ is the suspension of a based object $A''$  which is a relative CW--complex in $\CW{k+r-2,2r+k-2}{D}$.
\end{itemize}
\end{lem}

The presentations in Lemma \ref{k,2r+k} form the subcategory $\twist{k,r}{D}$ of $\twistD$ consisting of presentations $\partial_X: X'' \rightarrow X$, where $X$ is a $D$--cone in $\CW{k,k+r}{D}$ and $X''$ is a suspension in $\CW{k+r-1,2r+k-1}{D}$.

To discuss questions of realizability for a functor $\lambda: \A \rightarrow \B$, we consider pairs $(A, b)$, where $b: \lambda A \cong B$ is an equivalence in $\B$. Two such pairs are equivalent, $(A, b) \sim (A', b')$, if there is an equivalence $g: A \rightarrow A'$ in $\A$ with $\lambda g = b^{-1}b'$. The classes of this equivalence relation form the class of $\lambda$--realizations of $B$,
\[\text{Real}_{\lambda}(B) = \quo{$\{ (A, b) \, | \, b: \lambda A \equiv B \}$}{$\sim$}.\]
We say that $B$ is \emph{$\lambda$--realizable}, if $\text{Real}_{\lambda}(B)$ is non--empty. The functor $\lambda: A \rightarrow B$ is {\emph{representative} if all objects $B$ in $\B$ are $\lambda$--realizable. Further, we say that $\lambda$ \emph{reflects isomorphisms}, if a morphism $f$ in $\A$ is an equivalence whenever $\lambda(f)$ is an equivalence in $\B$. A morphism $\overline f: \lambda(A) \rightarrow \lambda(A')$ in $\B$ is \emph{$\lambda$--realizable}, if there is a morphism $f: A \rightarrow A'$ in $\A$ with $\lambda(f) = \overline f$. The functor $\lambda$ is \emph{full} if every morphism $f: A \rightarrow A'$ in $\A$ is $\lambda$--realizable. A functor which is full and representative and reflects isomorphisms is called a \emph{detecting functor}. 

\begin{thm}\label{presexist}
There is a functor
\[\tau: \CW{k,2r+k}{D} \longrightarrow \quo{$\twist{k,r}{D}$}{$\simeq_E$,}\]
which takes a relative CW--complex $(X, D)$ to the presentation in Lemma \ref{k,2r+k}. Morevover, $\tau$ is a detecting functor. 
\end{thm}

Theorem \ref{presexist} follows from VII 3.1 in \cite{AH}.

In general, the functor $\tau$ in Theorem \ref{presexist} is a linear extension of categories, not necessarily an equivalence.

\begin{cor}\label{presexistcor}
There is a 1--1 correspondence between homotopy types in $\CW{k,2r+k}{D}$ and isomorphism types of presentations in $\quo{$\twist{k,r}{D}$}{$\simeq_E$}$.
\end{cor}

Corollary \ref{presexistcor} shows that certain homotopy types of relative CW--complexes can be represented by presentations in the homotopy category of $\twist{}{D}$. In the remaining sections we provide examples where these presentations are described equivalently by algebraic models.

Now we consider the ``kernel'' of the functor $\tau$ in Theorem \ref{presexist}. Let $X_2$ and $Y_2$ be objects in $\CW{k,2r+k}{D}$ with presentations $\partial_X$ and $\partial_Y$ as in Lemma \ref{k,2r+k}. Then \ref{k,2r+k} yields the coaction map $\mu: X_2 \rightarrow X_2 \vee \Sigma X''$ and the action $+$. The functor $\tau$ of Theorem \ref{presexist} gives rise to the function
\[\tau: [X_2, Y_2]^D \longrightarrow [\partial_X, \partial_Y],\]
where the right hand side is the set of morphisms in $\quo{$\twist{k,r}{D}$}{$\simeq_E$}$. As $\tau$ is full, the function is surjective.

\begin{add}\label{orbits}
Take objects $X_2$ and $Y_2$ in $\CW{k,2r+k}{D}$ with presentations $\partial_X$ and $\partial_Y$ as in Lemma \ref{k,2r+k}, and $f, g \in [X_2,Y_2]^D$. Then $\tau(f) = \tau(g)$ if and only if $f = g + i_{\ast}\alpha$ for some $\alpha \in [\Sigma X'',Y]^D$, where $i: Y \rightarrow Y_2$ is the inclusion.
\end{add}

The isotropy groups, $J_g$, of the action $+$ of $\alpha$ can be computed by V (5.7) in \cite{AH}. We obtain the natural system $\quo{$[\Sigma X'',Y]$}{$J_g$}$, so that $\tau$ is a linear extension of categories, see \cite{AH}.

\section{Crossed Modules and Cells in Dimensions $1$ and $2$}
Putting $k = 0$ and $r = 1$ in Theorem \ref{presexist}, we obtain the detecting functor
\[ \tau: \CW{0,2}{D} \longrightarrow \quo{$\twist{0,1}{D}$}{$\simeq_E$}.\]
Here the right hand side is determined by the category $\CW{1}{D}$ which depends only on the $1$--type of $D$. In Section \ref{ncells} we computed the category $\CW{1}{D}$ in terms of cofibrant objects in $\grd{\G}$, where $\G$ is a groupoid representing the $1$--type of $D$.  Considering $\G$ as a crossed module concentrated in degree $1$ and using the definition of $\twist{}{}$, we now compute the category $\CW{0,2}{D}$ in terms of cofibrant objects in the category of crossed modules under $\G$. 

Cofibrations in the category of crossed modules over groupoids are morphisms in $\cross{}{}$ which are free extensions in degrees $1$ and $2$, see \cite{4dim}. We denote the subcategory of cofibrant objects $\G \rightarrowtail M$ in $\cross{}{\G}$ with $\text{Ob}(M) = \text{Ob}(\G)$ by $\cross{c}{\G}$. Morphisms $f, g: M \rightarrow M'$ in $\cross{}{\G}$ are homotopic if there is an $f$--crossed homomorphism $\alpha: M_1 \rightarrow M_2'$ such that $\alpha|_{\G} = 0, -f_1 + g_1 = d_{M'} \alpha$ and $-f_2 + g_2 = \alpha d_M$.

\begin{thm}\label{CW02}
Let $\G$ be a groupoid representing the $1$--type of $D$. Then there is a detecting functor 
\[ \tau: \CW{0,2}{D} \longrightarrow \quo{$\cross{c}{\G}$}{$\simeq$}.\]
\end{thm}

If $D = \ast$ is a point, $\tau$ is an equivalence of categories and Theorem \ref{CW02} yields a classical result due to J.H.C. Whitehead \cite{Whitehead3}. The corresponding result for $\CW{0,2}{D,\ast}$, with $D$ connected and pointed, is discussed in VI.3 of \cite{AH}.

\section{Crossed Complexes and Cells in Dimensions $2$ and $3$}
In this Section we apply Theorem \ref{presexist} to the homotopy category, $\CW{1,3}{D}$, of relative CW--complexes with cells in dimension two and three only. Lemma \ref{CWnDisfull} implies that a relative CW--complex, $(X,D)$, in $\CW{1,3}{D}$ is equivalent to a $D$--cone. However, here we use Lemma \ref{k,2r+k} which guarantees a presentation $\partial_X$ of $X$ in $\CW{2}{D}$, rather than the $D$--cone structure. 

By Lemma \ref{CWnDisfull}, $\CW{2}{D}$ is a full subcategory of $\coneD$. Hence we may consider presentations in $\CW{2}{D}$, that is, presentations $\partial_X: X'' \rightarrow X$ with $X$ in $\CW{2}{D}$ and $X''$ a suspension in $\CW{2}{D}$. Such presentations correspond to attaching maps of $3$--cells of a CW--complex in $\CW{1,3}{D}$ and have algebraic models due to the equivalence $\tau$ in Theorem \ref{CW2dcrossceq}. We recall the notion of \emph{crossed complexes over groupoids of dimension $3$} to describe these algebraic models.

A \emph{$3$--dimensional crossed complex over the groupoid $M_1$} is a sequence
\[ \xymatrix{
M_3 \ar[r]^{\partial_3} &  M_2 \ar[r]^{\partial_2} & M_1,}\]
such that
\begin{enumerate}
\item $\partial_2: M_2 \rightarrow M_1$ is a crossed module over groupoids; 
\item $M_3$ is a totally disconnected abelian groupoid with $\text{Ob}(M_3) = \text{Ob}(M_1)$ and
\item $\partial_3$ is a functor which is the identity on objects, such that $\partial_2 \partial_3 = 0$ and $\text{im}(\partial_2)$ acts trivially on $M_3$.
\end{enumerate}

A \emph{morphism of $3$--dimensional crossed complexes}, $f: M \rightarrow N$, is a sequence of morphisms $f_i: M_i \rightarrow N_i$, $i = 1, 2, 3$, of groupoids inducing the same map on objects and compatible with the boundary maps and the actions of both crossed chain complexes. A morphism $f$ of $3$--dimensional crossed complexes is a {\emph{cofibration} if $f$ is a free extension in every degree, see \cite{4dim}. Given a crossed module $\D$ representing the $2$--type of the space $D$, we consider cofibrant objects $q: \D \rightarrowtail M$ in the category of $3$--dimensional crossed complexes under $\D$ of the form
\[\xymatrix{
& \D_2 \ar[r] \ar[d]_{q_2} & \D_1 \ar@{=}[d]^{q_1} \\
M_3 \ar[r]_{\partial_3} & M_2 \ar[r]_{\partial_2} & M_1,}\]
where $q_1$ is the identity, and denote the category of such cofibrant objects by $\cross{3,c}{\D}$. Morphisms $f, g: M \rightarrow M'$ in $\cross{3,c}{\D}$ are \emph{homotopic} if there is an $f_1$--equivariant homomorphism of modules $\alpha: M_2 \rightarrow M_3'$ such that 
\begin{eqnarray*}
\alpha q_2 & = & 0,  \\
-f_2 + g_2 & = & \partial_3' \alpha, \\
-f_3 + g_3 & = & \alpha \partial_3,
\end{eqnarray*}
and we denote the resulting homotopy category by $\quo{$\cross{3,c}{\D}$}{$\simeq$}$.

\begin{thm}\label{CW13det}
Let $\D$ be a crossed module representing the $2$--type of $D$ with $\D_1$ a free groupoid. Then there is a detecting functor
\[\tau: \CW{1,3}{D} \longrightarrow \quo{$\cross{3,c}{\D}$}{$\simeq$}.\]
\end{thm}

\begin{proof}
A presentation $\partial_X$ in $\CW{2}{D}$ corresponds to an attaching map of $3$--cells of a relative CW--complex $(X, D)$ in $\CW{1,3}{D}$ and is equivalently given by the boundary $\partial_3$. Morphisms between such presentations are principal and correspond to maps between crossed complexes.
\end{proof}

Note that there is a corresponding detecting functor for $\CW{1.3}{D,\ast}$ with $D$ pointed and connected involving crossed complexes over groups.

\begin{obs}
There is a result similar to Theorem \ref{CW13det} for $\CW{0,3}{D,\ast}$. In this case $q_1$ above is not the identity, but a free extension.
\end{obs}

\section{Quadratic Modules and Cells in Dimensions $2$, $3$ and $4$}
In this Section we provide a detecting functor for the category $\CW{1,4}{D,\ast}$ of pointed relative CW--complexes with cells in dimensions $2$, $3$ and $4$ only. The attaching map of $4$--cells yields a presentation $\partial_X$ in the category $\CW{1,3}{D,\ast}$ which is a subcategory of $\coneD$ by Lemma \ref{k,k+r}. The algebraic description of $\CW{1,3}{D,\ast}$ in terms of quadratic modules in Section \ref{quadmod} leads to an algebraic description of the presentation $\partial_X$ for which we require the notion of \emph{$4$--dimensional quadratic complexes}. Such a complex is a diagram
\[\xymatrix{
&& C_2 \otimes C_2 \ar[dl]_{\omega} \ar[d] & \\
Q_4 \ar[r]_{\partial_4} &  Q_3 \ar[r]_{\partial_3} & Q_2 \ar[r]_{\partial_2}  & Q_1,}\]
such that
\begin{enumerate}
\item $(\partial_2, \partial_3, \omega)$ is a quadratic module;
\item $Q_4$ is a $Q_1$--module such that $\partial_2(Q_2)$ acts trivially and
\item $\partial_4$ is a homomorphism of $Q_1$--modules, such that $\partial_3\partial_4 = 0$.
\end{enumerate}
A \emph{morphism of $4$--dimensional quadratic complexes}, $f: Q \rightarrow Q'$, is a sequence of morphisms $f_i: M_i \rightarrow N_i$, $1 \leq i \leq 4$, such that $(f_3, f_2, f_1)$ yields a morphism of quadratic complexes, $f_4$ is an $f_1$--equivariant homomorphism of modules and $\partial_4 f_4 = f_3 \partial_4$. A morphism of $4$--dimensional quadratic complexes is a \emph{cofibration} if it is a free extension in every degree, see \cite{4dim}. Given a quadratic module $\D$ representing the $3$--type of the connected and pointed space $D$, we consider the category of cofibrant objects $q: \D \rightarrowtail Q$ in the category of $4$--dimensional quadratic complexes under $\D$ of the form
\[\xymatrix{
& \D_3 \ar[r] \ar[d]_{q_3} & \D_2 \ar[r] \ar[d]^{q_2} & \D_1 \ar@{=}[d]^{q_1} \\
Q_4 \ar[r]_{\partial_4} & Q_3 \ar[r]_{\partial_3} & Q_2 \ar[r]_{\partial_2} & Q_1,}\]
where $q_1$ is the identity. We denote the category of such objects by $\qud{4,c}{\D}$. Morphisms $f, g: Q \rightarrow Q'$ in $\qud{4,c}{\D}$ are \emph{homotopic} if there there are functions $\alpha_2: Q_2 \rightarrow Q_3'$ and $\alpha_3: Q_3 \rightarrow Q_4'$ which determine a homotopy in the category of quadratic chain complexes as in IV.4 in \cite{4dim}, such that $\alpha_2q_2 = 0$ and $\alpha_3q_3 = 0$. We denote the resulting homotopy category by $\quo{$\qud{4,c}{\D}$}{$\simeq$}$ and provide an explicit definition of homotopy for the reduced case in the next section.

\begin{thm}\label{CW14}
Let $\D$ be a quadratic module representing the $3$--type of the pointed and connected space $D$, which is free in degrees $1$ and $2$. Then there is a detecting functor
\[\tau: \CW{1,4}{D,\ast} \longrightarrow \quo{$\qud{4,c}{\D}$}{$\simeq$}.\]
\end{thm}

The homotopies in $\qud{4,c}{\D}$ are given by $\alpha_2$ and $\alpha_3$ as above, with $\alpha_2$ arising from Theorem \ref{CW13}, and $\alpha_3$ corresponding to the homotopy $\simeq_E$ in $\twist{}{D}$ according to Theorem \ref{presexist}. For $D = \ast$, the homotopy types of $\CW{1,4}{\ast}$ were classified by J.H.C. Whitehead in \cite{Whitehead1} and \cite{Whitehead3}.

\begin{obs}
There is a result similar to Theorem \ref{CW14} for $\CW{0,4}{D,\ast}$. In this case $q_1$ above is not the identity, but a free extension. For a proof one needs the theory in \cite{4dim}.
\end{obs}

\section{Reduced Quadratic Modules}
Reduced quadratic modules, that is, quadratic modules $Q$ with $Q_1  = 0$, are algebraic models of simply connected $3$--types. We recall the precise definition. A \emph{reduced quadratic module}, $Q = (\omega, \partial_3)$, is a diagram
\[ \xymatrix{ 
C \otimes C \ar[r]^{\omega} & Q_3 \ar[r]^{\partial_3} & Q_2,}\]
in $\gr{}$, such that
\begin{enumerate}
\item $Q_2$ is a $\text{nil}(2)$--group (that is, triple commutators in $Q_2$ are trivial) with $C = Q_2^{ab}$, where $Q^{ab}$ is the abelianization of $Q_2$ with quotient map $Q_2 \twoheadrightarrow Q_2^{ab}, x \mapsto \{x\}$;
\item The composition $\partial_3 \omega$ is the commutator map, that is, for $x, y \in Q_2$, \[\partial_3 \omega (\{x\} \otimes \{y\}) = -x-y+x+y = (x,y);\]
\item For $q \in Q_3$ and $x \in Q_2$, \[ \omega(\{\partial_3q\} \otimes \{x\} + \{x\}\otimes \{\partial_3q\}) = 0 \quad \text{and}\]
\item Commutators in $Q_3$ satisfy, for $p, q \in Q_3$, \[(p,q) = -p-q+p+q = \omega(\{\partial_3p\}\otimes\{\partial_3q\}).\]
\end{enumerate}
A \emph{morphism}, $f$, of reduced quadratic modules is a pair of homomorphisms, $f_i: Q_i \rightarrow Q_i', i = 1, 2$, such that $f_2\partial_3 = \partial_3'f_3$ and $q_3 \omega = \omega' (f_2^{ab} \otimes f_2^{ab})$, where $f_2^{ab}$ is the homomorphism induced by $f_2$ on the abelianizations. This defines the category $\rqud{}{}$ of reduced quadratic modules. 

A \emph{$4$--dimensional reduced quadratic complex}, $Q$, is a commutative diagram
\[\xymatrix{
&& C \otimes C \ar[dl]_{\omega} \ar[d] \\
Q_4 \ar[r]_{\partial_4} & Q_3 \ar[r]_{\partial_3} & Q_2}\]
in $\gr{}$, such that $Q_4$ is an abelian group, $(\partial_3, \omega)$ is a reduced quadratic module with $C = Q_2^{ab}$ and $\partial_3\partial_4 = 0$. A \emph{morphism} $f$ of reduced quadratic modules is given by group homomorphisms $f_i: Q_i \rightarrow Q_i', i = 2, 3, 4$, such that $(f_3, f_2)$ is a morphism of reduced quadratic modules and $f_3\partial_4 = \partial_4'f_3$.  

Given a reduced quadratic module $\D$, we may view $\D$ as a $4$--dimensional reduced quadratic complex concentrated in degrees $2$ and $3$. Defining {\emph{cofibrations} in $\rqud{}{}$ via free extensions, similarly to IV.(C.4) in \cite{4dim},  we obtain the category $\rqud{4,c}{\D}$ of $4$--dimensional cofibrant reduced quadratic complexes under $\D$. Two morphisms $f, g: Q \rightarrow Q'$ in $\rqud{4,c}{\D}$ are \emph{homotopic} if there is a group homomorphism $\alpha_3: Q_3 \rightarrow Q_4'$ and a function  $\alpha_2: Q_2 \rightarrow Q_3'$, such that $\alpha_2|_{\D_2} = 0, \alpha_3|_{\D_3} = 0$, 
\begin{eqnarray}\label{homotopy1}
-f_2 + g_2 & = & \partial_3 \alpha_2,\\
-f_3 + g_3 & = & \partial_4' \alpha_3 + \alpha_2\partial_3,\\
-f_4 + g_4 & = & \alpha_3 \partial_4
\end{eqnarray}
and
\begin{equation}\label{homotopy2}
\alpha_2(x+y) = \alpha_2 x + \alpha_2 y + \omega'(\{-f_2x + g_2x\} \otimes \{f_2 y\}).
\end{equation}
We denote the resulting homotopy category by $\quo{$\rqud{4,c}{\D}$}{$\simeq$}$. As $\rqud{4,c}{\D}$ is a subcategory of $\qud{4,c}{\D}$, Theorem \ref{CW14} yields the following result as a special case.

\begin{thm}\label{CW14ast}
Let $\D$ be a reduced quadratic module representing the $3$--type of the pointed simply connected space $D$ and assume that $\D$ is free in degree $2$. Then there is a detecting functor
\[\tau: \CW{1,4}{D,\ast} \longrightarrow \quo{$\rqud{4,c}{\D}$}{$\simeq$}.\]
\end{thm}

\section{Application: Self--Maps of $S^2\times S^2$ Fixing the Diagonal}
We conclude with an example simple enough to be described within the limitations of this paper but using the theory developed in an essential way. Consider the product $S^2 \times S^2$ of two $2$--dimensional spheres with diagonal $\Delta: D = S^2 \rightarrowtail S^2 \times S^2$. 

\begin{thm}\label{S2crossS2}
There are exactly 16 essential self--maps of $S^2 \times S^2$ fixing the diagonal, $D$, that is, the set $[S^2 \times S^2, S^2 \times S^2]^D$ of homotopy classes relative $D$ of self--maps of $S^2 \times S^2$ has exactly 16 elements.
\end{thm}

The remainder of this Section is devoted to the proof of Theorem \ref{S2crossS2}.

There are four canonical self--maps of $S^2 \times S^2$ which fix the diagonal, namely the identity, $I$, the interchange map, $T$, $P' = \Delta \circ \text{pr}_1$ and $P'' = \Delta \circ \text{pr}_2$, where $\text{pr}_i: S^2 \times S^2 \rightarrow S^2$ is the projection onto the $i$--th factor for $i = 1, 2$. These maps form a monoid $\M$ with multiplication table
\[\begin{array}{|c||c|c|c|c|}
\hline 
\phantom{x} & I & T & P' & P'' \\ \hline \hline
I & I & T & P' & P'' \\ \hline 
T & T & I & P' & P'' \\ \hline 
P' & P' & P'' & P' & P'' \\ \hline 
P'' & P'' & P' & P' & P'' \\ \hline
\end{array}\]
The homotopy group $\pi_4(S^2 \times S^2) = \Z_2 \oplus \Z_2$ is an $\M$--bimodule with action given by
\[\begin{array}{lll}
T^{\ast}(x,y) = (x,y) & T_{\ast}(x,y) = (y,x) \\
P'^{\ast}(x,y) = 0 & P'_{\ast}(x,y) = (x,x) \\
P''^{\ast}(x,y) = 0 & P''_{\ast}(x,y) = (y,y)
\end{array}\]
for $(x,y) \in \Z_2 \oplus \Z_2$. We define the monoid $\overline{\M} = \M \times (\Z_2 \oplus \Z_2)$ by the multiplication
\[(m,(x,y))\circ (m',(x',y')) = (mm',(m_{\ast}(x',y') + m'^{\ast}(x,y)),\]
that is, $\overline{\M}$ is a split linear extension of $\M$.

\begin{thm}
The set $[S^2\times S^2, S^2 \times S^2]^D$ together with composition of maps is a monoid ismomorphic to $\overline{\M}$. The group of homotopy equivalences under $D$, $\text{Aut}(S^2 \times S^2)^D$, is the group of units in $\overline{\M}$, given by the semi--direct product $\Z_2 \times (\Z_2 \oplus \Z_2)$, with $\Z_2$ generated by $T$.
\end{thm}

Choosing a point $\ast \in S^2 \times S^2$ in the complement of the diagonal, there is a closed ball $B^4$ such that
\[\ast \in B^4 \subset S^2 \times S^2 \setminus \Delta(D).\]
Now pinch the boundary of $B^4$ to obtain the coaction map
\[\mu: S^2 \times S^2 \longrightarrow (S^2 \times S^2) \vee S^4.\]
For $\alpha: S^4 \rightarrow S^2 \times S^2$ and a self--map $f$ of $S^2 \times S^2$ fixing the diagonal, the map $f+\alpha = (f,\alpha)\mu$ is again a self--map of $S^2 \times S^2$ fixing the diagonal. Thus $\mu$ yields an action of $\pi_4(S^2 \times S^2)$ on $[S^2\times S^2, S^2 \times S^2]^D$. We consider the equivariant forgetful map
\[[S^2\times S^2, S^2 \times S^2]^D \longrightarrow [S^2\times S^2, S^2 \times S^2]^{\ast}.\]
Note that the action is free on the right hand side, as elements $\alpha \in \pi_3(S^2)$ and $\beta \in \pi_2(S^2)$ have trivial Whitehead product $[\alpha,\beta] = 0$, see (3.3.15) in \cite{OT}. Thus $\pi_4(S^2 \times S^2)$ also acts freely on $[S^2\times S^2, S^2 \times S^2]^D$. Hence the inclusion $\M \subset [S^2\times S^2, S^2 \times S^2]^D$ of monoids implies that there is an inclusion $\overline{\M} \subset [S^2\times S^2, S^2 \times S^2]^D$ of monoids. It remains to show that this inclusion is, in fact, an isomorphism of monoids and here the theory described in this paper plays an essential r{\^o}le.

Given a self--map $f$ of $S^2 \times S^2$ fixing the diagonal, we study the properties of the map $H_{\ast}(f)$ induced on the homology groups $H_{\ast} = H_{\ast}(S^2 \times S^2)$. Let $e_4 \in H_4 = \Z$ and $e', e'' \in H_2 = \Z \oplus \Z$ be the generators and suppose $H_4(f)(e_4) = k e_4$ with $k \in \Z$. As $H_{\ast}(f)$ is compatible with the intersection form, which maps $e_4$ to $\omega = e'\otimes e'' + e'' \otimes e' \in H_2 \otimes H_2$, we conclude $(H_2(f)\otimes H_2(f))(\omega) = k \omega$. On the other hand, $H_2(f)(e'+e'') = e'+e''$ or $H_2(f)(e') = -H_2(e'') + e'+e''$, as $f$ fixes the diagonal. Putting $H_2(f)(e') = ae' + be''$ with $a, b \in \Z$, a simple calculation yields
\begin{eqnarray*}
2a(1 - a) & = & 0 \\
2b(1 - b) & = & 0 \\
a + b - 2ab & = & k.
\end{eqnarray*}
Thus $a, b \in \{0, 1\}$ and hence $H_2(f) = m_{\ast}$ for some $m \in \M$. 

The product structure of $S^2 \times S^2$ implies
\[ [S^2 \times S^2, S^2 \times S^2]^D = [S^2 \times S^2, S^2]^D \times [S^2 \times S^2, S^2]^D.\]
We now replace $(S^2 \times S^2, D)$ by a simpler relative CW--complex. There is a homotopy $H: \Delta \simeq \mu$ from the diagonal $\Delta$ to the comultiplication $\mu: S^2 \rightarrow S^2 \vee S^2 \subset S^2 \times S^2$. Let $Z$ be the mapping cylinder of $\mu$. Then $H$ yields a map $\overline H: Z \rightarrow S^2 \times S^2$ under $D$, which is a homotopy equivalence in $\textbf{Top}$. By Corollary II 2.21 in \cite{AH}, $\overline H$ is also a homotopy equivalence in $\textbf{Top}^D$ and hence
\[[Z, S^2]^D = [S^2 \times S^2]^D.\]
As $Z$ is a relative CW--complex with cells in dimensions $2, 3$ and $4$, we may apply Theorem \ref{CW14ast} which computes the orbits of the action of $\pi_4(S^2)$ on $[Z, S^2]^D$ by Theorem \ref{orbits}. We must show that the set of orbits, $\tau[Z, S^2]^D$, has exactly two elements.

The functor $\tau$ maps a homotopy class in $[Z, S^2]^D$ to a retraction $Q \rightarrow \D$ of the cofibration $\D \rightarrowtail Q$ associated with the relative CW--complex $(Z, D)$. The reduced quadratic module $\D$ of $D = S^2$ is given by $\D_2 = \Z = \langle e \rangle, \D_3 = \Z = \langle \omega(e \otimes e) \rangle, \partial_3 = 0$ and $\omega$ an isomorphism. The group $Q_2$ is the free $\text{nil}(2)$--group with generators $e, e'$ and $e''$ corresponding to the $2$--cells of $Z$. For $e_3 \in Q_3$ and $e_4 \in Q_4$ corresponding to the $3$-- and $4$--cell of $Z$, respectively, we obtain $\partial_3(e_3) = -e + e' + e''$ and $\partial_4(e_4) = \omega(e' \otimes e'' + e'' \otimes e')$. Together with the cofibration $\D \rightarrowtail Q$, this determines $Q$ completely.

We must show that there are two homotopy classes of retractions $g: Q \rightarrow \D$ of $\D \rightarrowtail Q$. By the calculations above, $H_2(g) = H_2(\text{pr}_i)$ with $i = 1$ or $2$. First we assume $i = 1$ and show that any retraction $g$ with $H_2(g) = H_2(\text{pr}_i)$ must be homotopic under $\D$ to the retraction $f: Q \rightarrow \D$ defined by $f_2 =  \tau(\text{pr}_1)_2$ and $f_3(e_3) = 0$. As $H_2(g) = H_2(\text{pr}_1)$ implies $g_2 = \tau(\text{pr}_1)_2$, we obtain $f_2(e') = e$ and $f_2(e'') = 0$. Further, $f_2(e) = e$, since $f$ is a retraction. Hence 
\[ f_2 \partial_3(e_3) = f_2 (-e + e' + e'') = -f_2(e) + f_2(e') + f_2(e'') =  0 = \partial_3 f_3(e_3),\]
showing that $f$ is compatible with boundary maps, and hence a well--defined retraction of $\D \rightarrowtail Q$. Now let $g: Q \rightarrow \D$ be a retraction of $\D \rightarrowtail Q$ with $H_2(g) = H_2(f) = H_2(\text{pr})_1$. Then $g_2 = f_2$ and thus a map $\alpha_2: Q_2 \rightarrow \D_3$ defining a homotopy $\alpha: f \simeq g$ under $\D$ must be a homomorphism by (\ref{homotopy2}) ($\alpha_3 = 0$ as $f_4 = g_4 = 0$). If $g_3(e_3) = r \omega(e \otimes e)$ for $r \in \Z$, put $\alpha_2(e) = 0, \alpha_2(e') = r \omega(e \otimes e)$ and $\alpha_2(e'') = 0$. Then $-f_2 + g_2 = 0 = \partial_3 \alpha_2$ as $\partial_3 = 0$ in $\D$. Since $\partial_3 = 0$ in $\D$ and $\alpha_2$ is a homomorphism, we obtain 
\[ (\partial_3 \alpha_2 + \alpha_2 \partial_3)(e_3) = \alpha_2\partial_3(e_3) = \alpha_2(-e + e' + e'') = r \omega(e \otimes e) = (-f_2 + g_2)(e_2).\]
Finally, $-f_4 + g_4 = 0 = \alpha_3 \partial_4$, so that $\alpha: f \simeq g$ by (\ref{homotopy1}).

The case $i=2$ is treated analogously, and we conclude that $\tau[Z, S^2]^D$ contains exactly two elements, namely the homotopy classes $\{ \tau(\text{pr}_1)\}$ and $\{\tau(\text{pr}_2)\}$.

\end{document}